\documentclass[12pt]{amsart}
\usepackage{amsmath}
\usepackage{amsxtra}
\usepackage{amscd}
\usepackage{amsthm}
\usepackage{amsfonts}
\usepackage{amssymb}
\usepackage {pstricks}
\usepackage{pstricks,pst-node}

\newtheorem{thm}{Theorem}[section]

\newtheorem{lem}[thm]{Lemma}

\newtheorem{cor}[thm]{Corollary}
\newtheorem{prop}[thm]{Proposition}

\theoremstyle{definition}
\newtheorem{definition}[thm]{Definition}

\theoremstyle{remark}

\numberwithin{equation}{section}

\begin{document}
\normalfont
%Referring commands:
\newcommand{\thmref}[1]{Theorem~\ref{#1}}
\newcommand{\secref}[1]{Section~\ref{#1}}
\newcommand{\lemref}[1]{Lemma~\ref{#1}}
\newcommand{\propref}[1]{Proposition~\ref{#1}}
\newcommand{\corref}[1]{Corollary~\ref{#1}}
\newcommand{\remref}[1]{Remark~\ref{#1}}
\newcommand{\eqnref}[1]{(\ref{#1})}
\newcommand{\exref}[1]{Example~\ref{#1}}

%Simplified symbols:
\newcommand{\nc}{\newcommand}

\nc{\on}{\operatorname}

\nc{\Z}{{\mathbb Z}}
\nc{\C}{{\mathbb C}}
\nc{\R}{{\mathbb R}}
\nc{\bbP}{{\mathbb P}}
\nc{\bF}{{\mathbb F}}

\nc{\boldD}{{\mathbb D}}
\nc{\oo}{{\mf O}}
\nc{\N}{{\mathbb N}}
\nc{\bib}{\bibitem}
\nc{\pa}{\partial}
\nc{\F}{{\mf F}}
\nc{\CA}{{\mathcal A}}
\nc{\CE}{{\mathcal E}}
\nc{\CP}{{\mathcal P}}
\nc{\CO}{{\mathcal O}}
\nc{\CK}{{\mathcal K}}
%%%Gus' operators
\nc{\Ann}{\text{Ann}}
\nc{\Rad}{\text{Rad}}
\nc{\Res}{\text{Res}}
\nc{\Ind}{\text{Ind}}
\nc{\Ker}{\text{Ker}}
\nc{\id}{\text{id}}
%\nc{\dim}{\text{dim}}
%%%%%%%%%%%%%%%%%%

\nc{\be}{\begin{equation}}
\nc{\ee}{\end{equation}}

\nc{\rarr}{\rightarrow}
\nc{\larr}{\longrightarrow}
\nc{\al}{\alpha}
\nc{\ri}{\rangle}
\nc{\lef}{\langle}

\nc{\W}{{\mc W}}
\nc{\gam}{\ol{\gamma}}
\nc{\Q}{\ol{Q}}
\nc{\q}{\widetilde{Q}}
\nc{\la}{\lambda}
\nc{\ep}{\epsilon}

\nc{\g}{\mf g}
\nc{\h}{\mf h}
\nc{\n}{\mf n}
\nc{\bb}{\mf b}
\nc{\G}{{\mf g}}

\nc{\D}{\mc D}
\nc{\cE}{\mc E}
\nc{\CC}{\mc C}
\nc{\CH}{\mc H}
\nc{\CT}{\mc T}
\nc{\CI}{\mc I}
\nc{\CR}{\mc R}

\nc{\UK}{{\mc U}_{\CA_q}}

\nc{\CS}{\mc S}

\nc{\CB}{\mc B}

\nc{\Li}{{\mc L}}
\nc{\La}{\Lambda}
\nc{\is}{{\mathbf i}}
\nc{\V}{\mf V}
\nc{\bi}{\bibitem}
\nc{\NS}{\mf N}
\nc{\dt}{\mathord{\hbox{${\frac{d}{d t}}$}}}
\nc{\E}{\mc E}
\nc{\ba}{\tilde{\pa}}
\nc{\half}{\frac{1}{2}}

\def\smapdown#1{\big\downarrow\rlap{$\vcenter{\hbox{$\scriptstyle#1$}}$}}

\nc{\mc}{\mathcal}
\nc{\ov}{\overline}
\nc{\mf}{\mathfrak}
\nc{\ol}{\fracline}
\nc{\el}{\ell}
\nc{\etabf}{{\bf \eta}}
\nc{\zetabf}{{\bf
\zeta}}\nc{\x}{{\bf x}}
\nc{\xibf}{{\bf \xi}} \nc{\y}{{\bf y}}
\nc{\WW}{\mc W}
\nc{\SW}{\mc S \mc W}
\nc{\sd}{\mc S \mc D}
\nc{\hsd}{\widehat{\mc S\mc D}}
\nc{\parth}{\partial_{\theta}}
\nc{\cwo}{\C[w]^{(1)}}
\nc{\cwe}{\C[w]^{(0)}} \nc{\wt}{\widetilde}
\nc{\gl}{\mf gl}
\nc{\K}{\mf k}

%%%%%%%%%%%%%%%%%New Definitions%%%%%%%%%%%
\newcommand{\U}{{\rm{U}}}
\newcommand{\End}{{\rm{End}}}
\newcommand{\Hom}{{\rm{Hom}}}
\newcommand{\Lie}{{\rm{Lie}}}
\newcommand{\Uq}{{{\rm U}_q}}
\newcommand{\GL}{{\rm{GL}}}
\newcommand{\Sym}{{\rm{Sym}}}
\newcommand{\rk}{{\rm{rk}}}
\newcommand{\tr}{{\rm{tr}}}
\newcommand{\Rea}{{\rm{Re}}}
\newcommand{\rank}{{\rm{rank}}}
\newcommand{\im}{{\rm{Im}}}

\advance\headheight by 2pt

%new definitions
\nc{\fb}{{\mathfrak b}}
\nc{\fg}{{\mathfrak g}}

\nc{\fh}{{\mathfrak h}}
\nc{\fk}{{\mathfrak k}}

\nc{\fl}{{\mathfrak l}}
\nc{\fn}{{\mathfrak n}}

\nc{\fp}{{\mathfrak p}}
\nc{\fu}{u}

%\nc{\fS}{{\mathfrak S}}

\nc{\fS}{{\Sym}}

\nc{\fsl}{{\mathfrak {sl}}}
\nc{\fsp}{{\mathfrak {sp}}}
\nc{\fso}{{\mathfrak {so}}}
\nc{\fgl}{{\mathfrak {gl}}}

\nc{\A}{\mc A} \nc{\cF}{{\mathcal F}}

\nc{\cA}{{\mathcal A}}
\nc{\cP}{{\mathcal P}}
\nc{\cC}{{\mathcal C}}
\nc{\cU}{{\mathcal U}}
\nc{\cB}{{\mathcal B}}

\def\lr{{\longrightarrow}}
\def\inv{{^{-1}}}

\def\xl{{\hbox{\lower 2pt\hbox{$\scriptstyle \mathfrak L$}}}}

\nc{\bX}{{\mathbf X}} \nc{\bx}{{\mathbf x}} \nc{\bd}{{\mathbf d}}
\nc{\bdim}{{\mathbf dim}} \nc{\bm}{{\mathbf m}}

\title[On endomorphisms of quantum tensor space]{On endomorphisms of quantum tensor space}

\author{G.I. Lehrer and R.B. Zhang}
\address{School of Mathematics and Statistics,
University of Sydney, N.S.W. 2006, Australia}
\email{gusl@maths.usyd.edu.au, rzhang@maths.usyd.edu.au}
\date {23rd June, 2008}

\begin{abstract} We give a presentation of the endomorphism algebra
$\End_{\cU_q(\fsl_2)}(V^{\otimes r})$, where $V$ is the
$3$-dimensional irreducible module for quantum $\fsl_2$ over the
function field $\C(q^{\frac{1}{2}})$. This will be as a quotient of
the Birman-Wenzl-Murakami algebra
$BMW_r(q):=BMW_r(q^{-4},q^2-q^{-2})$ by an ideal generated by a
single idempotent $\Phi_q$. Our presentation is in analogy with the
case where $V$ is replaced by the $2$- dimensional irreducible
$\cU_q(\fsl_2)$-module, the BMW algebra is replaced by the Hecke
algebra $H_r(q)$ of type $A_{r-1}$, $\Phi_q$ is replaced by the
quantum alternator in $H_3(q)$, and the endomorphism algebra is the
classical realisation of the Temperley-Lieb algebra on tensor space.
In particular, we show that all relations among the endomorphisms
defined by the $R$-matrices on $V^{\otimes r}$ are consequences of
relations among the three $R$-matrices acting on $V^{\otimes 4}$.
The proof makes extensive use of the theory of cellular algebras.
Potential applications include the decomposition of tensor powers
when $q$ is a root of unity.
\end{abstract}
\maketitle

\section{Introduction and statement of results}
Let $q^{\frac{1}{2}}$ be an indeterminate over $\C$.
Our objective is to give a presentation of
$\End_{\cU_q(\fsl_2)}(V_q^{\otimes r})$, where $V_q$ is the
irreducible $3$-dimensional representation of quantum $\fsl_2$. This
presentation will be as a quotient of the Birman-Murakami-Wenzl
algebra $BMW_r(q):=BMW_r(q^{-4},q^2-q^{-2})$ by an ideal generated
by a single quasi-idempotent $\Phi_q$. The endomorphism algebra
therefore has the same presentation as $BMW_r(q)$, with the
additional relation $\Phi_q=0$. This paper is a sequel to
\cite{LZ2}, where the main results were conjectured, and by and
large we maintain the notation of that work. Our results may be
stated integrally, i.e. in terms of algebras over the ring
$\C[q^{\pm \frac{1}{2}}]$, and they therefore have the potential to
generalise to the situation where $q$ is a root of unity. We intend
to address that issue in a future work.

\subsection{General notation.}
%Let $q^{\frac{1}{2}}$ be an indeterminate,
Denote the function field $\C(q^{\frac{1}{2}})$ by $\CK$. Let
$\cU_q=\cU_q(\fsl_2)$ be the quantised universal enveloping algebra
of $\fsl_2$ over $\CK$, and write $V_q$ for the 3-dimensional
irreducible $\cU_q$ module. More generally, write $V(d)_q$ for the
irreducible $\cU_q$-module with highest weight $d\in\Z_{\geq 0}$, so
that $V_q=V(2)_q$. It is well known that there is a homomorphism

\be\label{braidhom}
\eta:\CK \CB_r\lr \End_{\cU_q}(V(d)_q^{\otimes r}):=E_q(d,r),
\ee
where $\CB_r$ is the $r$-string braid group, and $\eta$ maps the
$i^{\text{th}}$ generator of $\CB_r$ to the $R$-matrix acting on
the $i^{\text{th}}$ and $(i+1)^{\text{st}}$ factors of the tensor
power (cf. \cite[\S\S 3,4]{LZ2}).

It was shown in \cite{LZ1} that $\eta$ is surjective for all $d$ and $r$.
It is our purpose
to use this to give an explicit presentation of $E_q(2,r)$, analogous to the
celebrated presentation of the Temperley-Lieb algebra $E_q(1,r)$ as a quotient
of the Hecke algebra of type $A$ (\cite[Theorem 3.5]{LZ2}). The first step
is to identify a finite dimensional quotient of $\CK\CB_r$ through which
$\eta$ factors.

\subsection{The algebra $BMW_r(\CK)$.} It was shown in \cite[Theorem 4.4]{LZ2}
that $\eta$ factors through the finite dimensional algebra $BMW_r(\CK)$,
which we now proceed to define.

Let $\CA_{y,z}$ be the ring $\C[y^{\pm 1},z]$, where $y,z$ are
indeterminates. The BMW algebra $BMW_r(y,z)$ over $\CA_{y,z}$ is the
associative $\CA_{y,z}$-algebra with generators $g_1^{\pm
1},\dots,g_{r-1}^{\pm 1}$ and $e_1,\dots,e_{r-1}$, subject to the
following relations:

The braid relations for the $g_i$:
\begin{equation}\label{braidgi}
\begin{aligned}
g_ig_j&=g_jg_i\text{ if }|i-j|\geq 2\\
g_ig_{i+1}g_i&=g_{i+1}g_ig_{i+1} \text{ for }1\leq i\leq r-1;\\
\end{aligned}
\end{equation}
The Kauffman skein relations:
\begin{equation}\label{kauffman}
g_i-g_i^{-1}=z(1-e_i)\text { for all }i;
\end{equation}
The de-looping relations:
\begin{equation}\label{delooping}
\begin{aligned}
&g_ie_i=e_ig_i=ye_i;\\\
&e_ig_{i-1}^{\pm 1}e_i=y^{\mp 1}e_i;\\
&e_ig_{i+1}^{\pm 1}e_i=y^{\mp 1}e_i.\\
\end{aligned}
\end{equation}

The next five relations are easy consequences of the previous three.
\begin{eqnarray}
&&e_ie_{i\pm 1}e_i=e_i; \label{bmwtl}\\
&&(g_i-y)(g_i^2-zg_i-1)=0; \label{cubic}\\
&&ze_i^2=(z+y^{-1} -y)e_i,\quad  %\text{ where }\delta=1+\frac{y\inv -y}{z};
\label{esquared}\\
&&-yze_i=g_i^2-zg_i-1; \label{equadg}\\
&&yzg_ie_{i+1}e_i=zg_{i+1}^{-1} e_i. \label{combeg}
\end{eqnarray}

It is easy to show that $BMW_r(y,z)$ may be defined using the
relations (\ref{braidgi}), (\ref{delooping}), (\ref{cubic}) and
(\ref{equadg}) instead of (\ref{braidgi}), (\ref{kauffman}) and
(\ref{delooping}), i.e. that (\ref{kauffman}) is a consequence of
(\ref{cubic}) and (\ref{equadg}).

We shall require a particular specialisation of $BMW_r(y,z)$
to a subring $\CA_q$ of $\CK$, which is defined as follows.
Let $\CS$ be the multiplicative subset of
$\C[q, q^{-1}]$
generated by $[2]_q$, $[3]_q$ and $[3]_q-1$. Let $\CA_q:=\C[q,
q^{-1}]_\CS:=\C[q,q\inv, [2]_q\inv,[3]_q\inv, (q^2+q^{-2})\inv]$
be the localisation of $\C[q, q^{-1}]$ at $\CS$.

Now let $\psi:\C[y^{\pm 1},z]\lr\CA_q$ be the homomorphism
defined by $y\mapsto q^{-4}$, $z\mapsto q^2-q^{-2}$. Then
$\psi$ makes $\CA_q$ into an $\CA_{y,z}$-module, and the
specialisation $BMW_r(q):=\CA_q\otimes_{\CA_{y,z}} BMW_r(y,z)$
is the $\CA_q$-algebra with generators which we denote, by abuse of
notation, $g_i^{\pm 1},e_i$ ($i=1,\dots,r-1$) and relations
(\ref{braidgiq}) below, with the relations (\ref{extrarel}) being
consequences of (\ref{braidgiq}).

\begin{equation}\label{braidgiq}
\begin{aligned}
g_ig_j&=g_jg_i\text{ if }|i-j|\geq 2\\
g_ig_{i+1}g_i&=g_{i+1}g_ig_{i+1} \text{ for }1\leq i\leq r-1\\
g_i-g_i^{-1}&=(q^2-q^{-2})(1-e_i)\text { for all }i\\
g_ie_i&=e_ig_i=q^{-4}e_i\\
e_ig_{i-1}^{\pm 1}e_i&=q^{\pm 4}e_i\\
e_ig_{i+1}^{\pm 1}e_i&=q^{\pm 4}e_i.\\
\end{aligned}
\end{equation}

\begin{equation}\label{extrarel}
\begin{aligned}
e_ie_{i\pm 1}e_i&=e_i  \\
(g_i-q^2)(g_i+q^{-2})&=-q^{-4}(q^2-q^{-2})e_i\\
(g_i-q^{-4})(g_i-q^2)(g_i+q^{-2})&=0\\
e_i^2&=(q^2+1+q^{-2})e_i\\
g_{i+1}^{-1} e_i&=q^{-4}g_ie_{i+1}e_i.\\
%&&-yze_i=g_i^2-zg_i-1. \label{equadgq}
\end{aligned}
\end{equation}

We shall be concerned with the following two specialisations
of $BMW_r(q)$.

\begin{definition}\label{spec}  Let $\phi_q:\CA_q\lr \CK=\C(q^{\frac{1}{2}})$ be the
inclusion map, and let $\phi_1:\CA_q\lr\C$ be the $\C$-algebra homomorphism
defined by $q\mapsto 1$. Define the specialisations
$BMW_r(\CK):=\CK\otimes_{\phi_q}BMW_r(q)$, and
$BMW_r(1):=\C\otimes_{\phi_1}BMW_r(q)$.
\end{definition}

The following facts are proved in \cite[Lemma 4.2, Theorem
4.4]{LZ2}.
\begin{prop}\label{bmwfact}
\begin{enumerate}
\item The surjective homomorphism $\eta$ of (\ref{braidhom}) factors through
$BMW_r(\CK)$. That is, there is a surjective homomorphism
$$
\eta_q:BMW_r(\CK)\lr E_r(2,q),
$$
in which the generators $g_i$ are mapped to the $R$ matrices on factors $i,i+1$.
\item The specialisation $BMW_r(1)$ is isomorphic to the Brauer algebra $B_r(3)$.
\end{enumerate}
\end{prop}

\subsection{The quasi-idempotent $\Phi_q$.}

Following \cite{LZ2}, we introduce an element $\Phi_q$ of $BMW_r(q)$,
and state its principal properties. Our objective will be to show that
$\Phi_q\in BMW(\CK)$ generates the kernel of $\eta_q$. We begin with the following
elements of $BMW_r(q)$. Note that we regard $BMW_s(q)\subset BMW_{s+1}(q)$ in the
usual way; it is well known that the algebra $BMW_s(q)$ generated by the $e_i$ and $g_i$
with $1\leq i\leq s-1$ and the relations among those generators is a subalgebra
of $BMW_{s+1}(q)$. In terms of diagrams, this subalgebra is spanned by the
diagrams spanning $BMW_s(q)$, with an additional string joining the rightmost nodes.

Let $f_i= -g_i-(1-q^{-2})e_i +q^2$, and set
\begin{eqnarray} \label{defFq}
F_q= f_1 f_3.
\end{eqnarray}
We also define $e_{1 4}=g_3^{-1} g_1 e_2 g_1^{-1} g_3$ and $e_{1 2 3
4} = e_2 g_1 g_3^{-1} g_2 g_1^{-1} g_3$.

\begin{definition}\label{defn:phiq}
Maintaining the above notation, define the following element of
$BMW_4(q)\subseteq BMW_r(q)$:
\begin{eqnarray}\label{defPhiq}
\Phi_q &=&aF_q e_2 F_q - b F_q - c F_qe_2e_{14}F_q +d F_q e_{1 2 3 4} F_q,
\end{eqnarray}
where
\begin{eqnarray}\label{abcd}
\begin{aligned}
a&=1+(1-q^{-2})^2,\\
b&=1+(1-q^{2})^2+(1-q^{-2})^2,\\
c&= \frac{1 +(2+q^{-2})(1-q^{-2})^2 + (1+q^2)
(1-q^{-2})^4}{([3]_q-1)^2},\\
d&=(q-q^{-1})^2=q^2(a-1).\\
\end{aligned}
\end{eqnarray}
\end{definition}

The principal properties of these elements are summarised in the following
statement, which is \cite[Prop. 7.3]{LZ2}.

\begin{prop}\label{propsPhiq}
The elements $F_q, \Phi_q$ have the following properties:
\begin{enumerate}
\item $F_q^2= (q^2+q^{-2})^2 F_q$. \label{PropsPhiq1}
\item $e_i\Phi_q=\Phi_q e_i=0$ for $i=1,2,3$. \label{PropsPhiq2}
\item $\Phi_q^2 = -(q^2+q^{-2})^2 (1+(1-q^{2})^2+(1-q^{-2})^2 )
\Phi_q$. \label{PropsPhiq3}
\item $\Phi_q$ acts as $0$ on $V_q^{\otimes 4}$. \label{PropsPhiq4}
\end{enumerate}
\end{prop}

\subsection{The classical limit $q\lr 1$; the Brauer algebra.}\label{sect:class}
Let $B_r(\delta)$ be the Brauer algebra over a commutative ring $A$,
with $\delta\in A$; this may be defined as follows. It has generators
$\{s_1,\dots,s_{r-1};e_1,\dots,e_{r-1}\}$, with relations
$s_i^2=1,\; e_i^2=\delta e_i,\; s_ie_i=e_is_i=e_i$ for all $i$,
$s_is_j=s_js_i,\;s_ie_j=e_js_i,\;e_ie_j=e_je_i$ if $|i-j|\geq 2$,
and $s_is_{i+1}s_i=s_{i+1}s_is_{i+1},\; e_ie_{i\pm 1}e_i=e_i$ and
$s_ie_{i+1}e_i=s_{i+1}e_i,\; e_{i+1}e_is_{i+1}=e_{i+1}s_{i}$ for all
applicable $i$. We shall assume the reader is familiar with the
diagrammatic representation of a basis of $B_r(\delta)$, and how
basis elements are multiplied by concatenation of diagrams. In
particular, the group ring $A\Sym_r$ is the subalgebra of
$B_r(\delta)$ spanned by the diagrams with $r$ ``through strings'',
and the algebra contains elements $w\in\Sym_r$ which are appropriate
products of the $s_i$.

In this work we shall take $A=\C$ and $\delta=3$; the corresponding
Brauer algebra will be denoted by $B_r(3)$. In the identification
given by Proposition \ref{bmwfact}(2) of the specialisation
$BMW_r(q)\otimes_{\phi_1}\C$ with $B_r(3)$, $g_i\otimes 1$ corresponds
to $s_i$ and $e_i\otimes 1$ corresponds to $e_i\in B_r(3)$.
Accordingly, $F_q$ and $\Phi_q$ specialise respectively to
$F=(1-s_1)(1-s_3)$ and $\Phi=Fe_2F-F-\frac{1}{4}Fe_2e_{1,4}F$,
where $e_{1,4}=s_1s_3e_2s_3s_1$. From the relations in
Proposition \ref{propsPhiq} we see that $F^2=4F$, and $\Phi^2=-4\Phi$.

Let $V_1$ be the irreducible $U(\fsl_2)$-module. The following statement
may be found in \cite[\S 6]{LZ2}.

\begin{prop}\label{braueractn}
\begin{enumerate}
\item There is a surjective homomorphism
\be\label{classeta}
\eta:B_r(3)\lr\End_{U(\fsl_2)}(V_1^{\otimes r}):=E(r).
\ee
\item The kernel of $\eta$ contains $\Phi$.
\end{enumerate}
\end{prop}

\subsection{Statements}\label{statements} We are now able to state our main results.

\begin{thm}\label{mainq}
Let $\cU_q$ be the quantised enveloping algebra of $\fsl_2$ over the
field $\CK=\C(q^{\frac{1}{2}})$. Let $V_q$ be the irreducible
three-dimensional $\cU_q$-module (with highest weight $2$), and let
$\eta_q:BMW_r(\CK)\lr E_r(2,q):=\End_{\cU_q}(V_q^{\otimes r})$ be
the surjective homomorphism defined in Proposition \ref{bmwfact}
(1). Then the kernel of $\eta_q$ is the two sided ideal of
$BMW_r(\CK)$ generated by the quasi-idempotent $\Phi_q$, defined in
Definition \ref{defn:phiq}.

In particular, $E_r(2,q)$ has a presentation given as follows. $E_r(2,q)$ has
generators $g_i,e_i$, $i=1,\dots,r-1$, with relations (\ref{braidgiq}), plus
the relation $\Phi_q=0$.
\end{thm}

The classical ($q=1$) analogue of the above statement is as follows.

\begin{thm}\label{mainclass}
Let $U$ be the universal enveloping algebra of $\fsl_2(\C)$, and let $V_1$
be the irreducible $U$-module of dimension three. If $B_r(3)$ is the complex
$r$-string Brauer algebra with parameter $3$ (\S \ref{sect:class}) and
$\eta:B_r(3)\lr \End_{U(\fsl_2)}(V_1^{\otimes r}):=E(r)$ be the surjective homomorphism
of Prop. \ref{braueractn}. Then the kernel of $\eta$ is the two sided ideal of
$B_r(3)$ which is generated by the quasi-idempotent $\Phi$ of
\S \ref{sect:class}.
\end{thm}

It should be noted that both $BMW_r(\CK)$ and $B_r(3)$ are non-semisimple
if $r\geq 5$. This is part of the significance of these results.

We shall prove both theorems together, making use of the following fact.

\begin{prop}\label{redn:class}
Theorem \ref{mainq} is a consequence of Theorem \ref{mainclass}.
\end{prop}
\begin{proof}[Sketch of Proof] This is Corollary 7.12~(1) of \cite{LZ2}.
For the convenience of the reader, we give a brief sketch of the proof here.

The algebra $BMW_r(q)$ has a well known cellular structure
(\cite{X}), with the cells being indexed by the partially ordered
set \be\label{lambda-r} \Lambda_r:=\{{\text {partitions}}\;
\lambda\mid |\lambda|=t,\;0\leq t\leq r,\; t \equiv r (\mod 2)\},
\ee where for any partition
$\lambda=(\lambda_1\geq\lambda_2\geq\dots\geq \lambda_p>0)$, we
write $|\lambda|=\sum_i\lambda_i$. The partial order on $\Lambda_r$
is given by $\mu>\lambda$ if $|\mu|>|\lambda|$ or
$|\mu|=|\lambda|=t$ and $\mu>\lambda$ in the dominance order on
partitions of $t$. This cellular structure is inherited by the
specialisations $BMW_r(\CK)$ and $B_r(3)$ (cf. \cite{GL96}), and by
Prop. 7.1 of \cite{LZ2}. The cell modules, their radicals, and the
irreducible modules for any specialisation of $BMW_r(q)$ are
obtained by specialising the respective $BMW_r(q)$-modules.

Using this, one shows, using the general criteria established in
\cite[\S 5]{LZ2} that if $J_q$ is the ideal of $BMW_r(q)$ generated
by $\Phi_q$, and $J_\CK$ and $J$ are its respective specialisations
to $BMW_r(\CK)$ and $B_r(3)$, then $J_\CK$ contains the radical of
$BMW_r(\CK)$ if and only if $J$ contains the radical of $B_r(3)$.
But the latter are precisely the respective conditions for
$J_\CK$ (resp. $J$) to be the whole kernel of $\eta_q$ (resp. $\eta$),
whence the assertion.
\end{proof}

We therefore turn to the proof of Theorem \ref{mainclass}, which will require
detail concerning the cellular structure of the Brauer algebras. Before discussing
the details in \S\ref{sect:brauer}, we shall recall some basic facts concerning
cellular algebras.

\section{Cellular algebras and their radicals.} In this section we summarise the
main properties of cellular algebras of which we make use below. The
principal references are \cite{GL03,GL04} and \cite[\S 5]{LZ2}.

\subsection{Cellular algebras and cell modules}\label{ca:basic}
Let $\bF$ be any commutative ring,
which later will be taken to be a field. The $\bF$-algebra $B$ is {\it cellular} if
it has a {\it cell datum} $(\Lambda,M,C,{\,}^*)$ with the following properties
\begin{itemize}
\item $\Lambda$ is a finite
partially ordered set
\item $M:\Lambda\to\text{Sets}$ is a function which
associates a finite set $M(\lambda)$ to each $\lambda\in\Lambda$
\item $C$ is a function
$$
C:\amalg_{\lambda\in\Lambda}M(\lambda)\times M(\lambda)\lr B,
$$
whose image is an $\bF$-basis of $B$; for $S,T\in M(\lambda)$,
write $C(S,T):=C^\lambda_{S,T}$.
\item For any element $a\in B$ and $S,S'\in M(\lambda)$ there are scalars
$r_a(S',S)\in\bF$ such that for $S,T\in M(\lambda)$, we have
\be\label{cell-basic}
aC^\lambda_{S,T}\equiv\sum_{S'\in M(\Lambda)}r_a(S',S)C^\lambda_{S',T}\;\;
\mod B(<\lambda),
\ee
where $B(<\lambda)$ is the two-sided ideal of $B$ spanned by
$\{C^\mu_{U,V}\mid \mu<\lambda,\;\;U,V\in M(\mu)\}$.
\item The linear map defined by
${\,}^*:C^\lambda_{S,T}\mapsto {C^\lambda_{S,T}}^*:=C^\lambda_{T,S}$ is
an involutory anti-automorphism of $B$.
\end{itemize}

The {\it cell module} $W(\lambda)$ ($\lambda\in\Lambda$) is the free
$\bF$-module spanned by symbols $C(S)$ ($S\in M(\lambda)$), with the
$B$-action defined by (\ref{cell-basic}), i.e.
\[
a\cdot C(S)=\sum_{S'\in M(\lambda)}r_a(S',S)C(S'), \quad \text{for
$a\in B$}.
\]
Define the bilinear form $\phi_\lambda$ on the basis $\{C(S)\}$ of
$W(\lambda)$ by
$$
{C^\lambda_{S,T}}^2=\phi_\lambda(C(S),C(T))C^\lambda_{S,T}\;\;\mod B(<\lambda).
$$
The principal properties of $\phi_\lambda$ are:
\begin{itemize}
\item $\phi_\lambda$ is well-defined and is symmetric.
\item $\phi_\lambda$ is invariant, in the sense that for $u,v\in W(\lambda)$ and $a\in B$,
we have $\phi_\lambda(a\cdot u,v)=\phi_\lambda(u,a^*\cdot v)$.
\item It follows that $\Rad(\lambda):=\{w\in W(\lambda)\mid \phi_\lambda(w,v)=0\text{ for all }
v\in W(\lambda)\}$ is a submodule of $W(\lambda)$.
\item If $\bF$ is a field, then for any element $w\in W(\lambda)\setminus\Rad(\lambda)$,
we have $W(\lambda)=B\cdot w$. Hence $\Rad(\lambda)$ is the unique maximal submodule
of $W(\lambda)$.
\end{itemize}

The last statement above arises in the following context. In addition
to its structure as a left $B$-module defined above, the $\bF$-module
$W(\lambda)$ also has a right $B$-module structure, where the action
is as above, composed with ${\,}^*$. We define the
$\bF$-module monomorphism $C^\lambda$ by

\be\label{def:clambda}
\begin{aligned}
C^\lambda:&W(\lambda)\otimes_\bF W(\lambda)\lr B\\
&C(S)\otimes C(T)\mapsto C^\lambda_{S,T}.\\
\end{aligned}
\ee

The properties of this map are summarised as follows.
\begin{itemize}\label{props:clambda}
\item There is an $\bF$-module isomorphism
$B\lr\oplus_{\lambda\in\Lambda} \im (C^\lambda)$.
\item If we compose $C^\lambda$ with the
natural map $B\to B/B(<\lambda)$, where $B(<\lambda)$ is the
two-sided ideal defined above, we obtain a monomorphism of
$(B,B)$ bimodules, whose image is, in the obvious notation,
$B(\leq\lambda)/B(<\lambda)$.
\item For $u,v$ and $w$ in $W(\lambda)$, we have
$$
C^\lambda(u\otimes v)\cdot w=\phi_\lambda(v,w)u.
$$
\end{itemize}

If $\bF$ is a field, the cyclic nature of $W(\lambda)$ if $\phi_\lambda
\neq 0$ follows from the last statement.

\subsection{Simple modules and composition factors}

It follows from \S\ref{ca:basic} that for each $\lambda\in\Lambda$,
either $\phi_\lambda=0$ or $L(\lambda):=W(\lambda)/\Rad(\lambda)$
is a simple $B$-module. One of the keys to understanding these modules,
and more generally the composition factors of the cell modules $W(\lambda)$,
is the following fact.

\begin{prop}\label{homs} Let $\lambda,\mu\in\Lambda$, and
suppose we have a non-zero homomorphism $\psi:W(\mu)\to W(\lambda)/M$,
where $M$ is some submodule of $W(\lambda)$. Then
\begin{enumerate}
\item $\mu\geq\lambda$.
\item If $\mu=\lambda$ and $\bF$ is a field then $\psi$ is
realised as multiplication by a scalar.
\end{enumerate}
\end{prop}

Although neither is essential, we shall now make two simplifying
assumptions. First, we assume that {\it  henceforth $\bF$ is a field},
and secondly, that for each $\lambda\in\Lambda$, $\phi_\lambda\neq 0$.
This says that $B$ is quasi-hereditary. The algebras to which we shall apply
the theory satisfy both these assumptions, but it is possible to develop
the theory without them.

Given the assumptions, we have the following summary of the main points
concerning the representation theory of $B$.

\begin{itemize}
\item The modules $L(\lambda)=W(\lambda)/\Rad(\lambda)$ ($\lambda\in\Lambda$)
form a complete set of non-isomorphic simple $B$-modules.

\item If $L(\mu)$ is a composition factor of $W(\lambda)$, then $\mu\geq\lambda$.
\item The following statements are equivalent: (a) $B$ is semisimple;
(b) Each cell module $W(\lambda)$ is simple; (c) For each
$\mu\neq\lambda$, $\Hom_B(W(\mu),W(\lambda))=0$.

\end{itemize}

\subsection{The radical} Suppose $B$ is a cellular algebra over the
field $\bF$ as above, and denote the radical of $B$ by $\CR$. This is a
two sided ideal of $B$ which may be defined by the property that $B/\CR$
is semisimple, and any homomorphism $B\lr A$, where $A$ is a semisimple
$\bF$-algebra, factors through $B/\CR$. In particular, we have an exact
sequence of $\bF$-algebras
$$
0\lr \CR\lr B\overset{\pi}{\lr}
\oplus_{\lambda\in\Lambda}\End_\bF(L(\lambda))\lr 0,
$$
so that $\CR$ may be characterised as the set of elements of $B$ which act
trivially on each of the irreducible $B$-modules. Note that in the exact sequence
above, for each $\lambda\in\Lambda$, the subspace $\im (C^\lambda)$ of $B$ is mapped
by $\pi$ (surjectively) onto $\End_\bF(L(\lambda))$.

Now the key point of the proof of Theorem \ref{mainq} is to prove that a certain
two-sided ideal of $B$ contains $\CR$. With this in mind we state the
next result, which may be found in \cite[Lemma 5.3]{LZ2}.

\begin{prop}\label{prop:rad}
The radical $\CR$ has a filtration
($\CR(\lambda):=\CR\cap B(\leq \lambda)_{\lambda\in\Lambda}$)
with the property that there is a $(B,B)$ bimodule isomorphism
$$
\CR(\lambda)/(\CR\cap B(<\lambda))\lr W(\lambda)\otimes_\bF \Rad(\lambda)
+\Rad(\lambda)\otimes_\bF W(\lambda)\subseteq W(\lambda)\otimes_\bF W(\lambda).
$$
\end{prop}

This leads to the next result, which
is easily deduced from \cite[Theorem 5.4]{LZ2}, and which is
crucial in the proof of our main theorem.

\begin{thm}\label{thm:rad}
Maintaining the notation above, let $J$ be a two-sided ideal of $B$
which satisfies $J^*=J$. Define complementary subsets
$\Lambda^0$ and $\Lambda^1$ of $\Lambda$ by
$\Lambda^0:=\{\lambda\in\Lambda\mid J\cdot L(\lambda)=0\}$ and
$\lambda^1:=\Lambda\setminus\Lambda^0$. Then $J\supseteq \CR$ if
and only if, for each $\lambda\in\Lambda^0$, all composition factors of
the cell module $W(\lambda)$ other than $L(\lambda)$ are isomorphic to some
$L(\mu)$ for $\mu\in\Lambda^1$.
\end{thm}
\begin{proof}
It is a straightforward consequence of Proposition \ref{prop:rad} above and
the general characterisation of $\CR$ as the set of elements of $B$ which
annihilate each $L(\lambda)$, that $J\supseteq\CR$ if and only if,
for each $\lambda\in\Lambda$, $J\cdot W(\lambda)\supseteq \Rad(\lambda)$.
But for $\lambda\in\Lambda^1$, since $J\cdot W(\lambda)\not\subseteq \Rad(\lambda)$,
it follows from the cyclic nature of $W(\lambda)$ (see above) that
$J\cdot W(\lambda)=W(\lambda)\supset \Rad(\lambda)$.

It is therefore evident (cf. \cite[Theorem 5.4(3)]{LZ2}),
that $J\supseteq \CR$ if and only if
for each $\lambda\in\Lambda^0$, $J\cdot W(\lambda)\supseteq \Rad(\lambda)$.
We shall show that the condition in our statement is equivalent to this.

Take $\lambda\in\Lambda^0$. Let
$$
W(\lambda)=W_0\supset W_1=\Rad(\lambda)\supset W_2\supset\dots\supset W_s=0
$$
be a composition series of $W(\lambda)$, and write $L_i$ for the simple quotient
$L_i:=W_{i-1}/W_i$ ($i=1,\dots,s$). Then $L_1\cong L(\lambda)$ and
since the $L(\mu)$ ($\mu\in\Lambda$) form a complete set of isomorphism classes
of simple $B$-modules, it follows from Proposition \ref{homs} that for $i>1$,
$L_i\cong L(\mu_i)$ for some $\mu_i\gneq\lambda$. Applying $J$, we obtain a sequence
$$
J\cdot W(\lambda)=J\cdot W_0\supseteq J\cdot W_1=J\cdot \Rad(\lambda)
\supseteq J\cdot W_2\supseteq \dots\supseteq J\cdot W_s=0,
$$
and the quotient $J\cdot W_{i-1}/J\cdot W_i\cong J\cdot L(\mu_i)$.

It is therefore clear that $J\cdot W(\lambda)\supseteq\Rad(\lambda)$ if and only
if, for each $i>1$, $J\cdot L(\mu_i)\neq 0$, i.e. $\mu_i\in\Lambda^1$.
\end{proof}
\section{On the representation theory of the Brauer algebras.}\label{sect:brauer}

We shall require some detailed analysis of the cell modules
of the Brauer algebras and their radicals. Most of the material in this section
is adapted from \cite{GL96,DHW,HW,LZ2}. Since there is no difference
in the arguments, for the purposes of the current section, $\CK$
could be any field of characteristic zero, and we shall consider
the Brauer algebras $B_r(\delta)$,
where $\delta$ is any non-zero element of $\CK$. This restriction is not
essential, but enables us to simplify the exposition; moreover our application
is to this situation ($\delta=3$). We shall write $B_r$ for $B_r(\delta)$
for this section only.

Notation will be similar to that in \cite{LZ2}. For $\lambda\in\Lambda_r$,
we have a cell module $W_r(\lambda)$ which has a canonical symmetric
bilinear invariant form $(-,-)_\lambda$, whose radical $\Rad_r(\lambda)$
is the unique maximal submodule of $W_r(\lambda)$. In our situation
($\delta\neq 0$) we have $(-,-)_\lambda\neq 0$ for all $\lambda\in\Lambda_r$
(the ``quasi-hereditary'' case). The irreducible quotients $L_r(\lambda):=
W_r(\lambda)/\Rad_r(\lambda)$ are pairwise non-isomorphic, and provide
a complete set of irreducible $B_r$-modules.

For the proof of the next Lemma, we shall require the following
explicit description of the cell modules. Recall that for any $t$,
the group algebra $\C\Sym_t$ is a subalgebra of $B_t$, realised as
the span of the diagrams with $t$ through strings (or generated by
the $s_i$). Let $\lambda\in\Lambda_r$ with $|\lambda|=t$, and
$r=t+2k$. Since $B_r$ is a $(B_r,B_r)$ bimodule, the subspace
$\tilde {I_r^t}:=B_re_{t+1}e_{t+3}\dots e_{t+2k-1}$ is a
$(B_r,B_t)$-bimodule, and hence {\it a foriori} a
$(B_r,\C\Sym_t)$-bimodule. Clearly the subspace ${\tilde {I_r^t}}'$
spanned by diagrams with at most $t-2$ through strings is a
sub-bimodule, to be interpreted as $0$ if $t<2$. Define $I_r^t$ as the quotient
$(B_r,\C\Sym_t)$-bimodule ${\tilde {I_r^t}}/{\tilde {I_r^t}}'$.

It is easy to see from the description
in \cite[Def. (4.9)]{GL96} (cf. also \cite{DHW}) that
\be\label{defn:w}
W_r(\lambda)\cong I_r^t\otimes_{\C\Sym_t}S(\lambda),
\ee
where $S(\lambda)$ is the cell (or Specht) module for $\C\Sym_t$ which corresponds
to the partition $\lambda$ of $t$.

\begin{lem}\label{strictineq}
Let $\mu_1,\mu_2\in\Lambda_r$. If $L_r(\mu_1)$ is a composition
factor of $W_r(\mu_2)$, then $|\mu_1|>|\mu_2|$ or $\mu_1=\mu_2$.
\end{lem}
\begin{proof}
Suppose $\phi:W_r(\mu_1)\lr W_r(\mu_2)/M$ is a $B_r$-homomorphism,
where $\mu_1,\mu_2\in\Lambda_r$ and $|\mu_1|=|\mu_2|=t$. We shall
show that if $\mu_1\neq \mu_2$, $\phi=0$.

Consider the subspaces $W_r^0(\mu_i)$ of $W_r(\mu_i)$ ($i=1,2$) defined by
$$W_r^0(\mu_i)=e_{t+1}e_{t+3}\dots e_{r-1}W_r(\mu_i).$$ Since
$e_{t+1}e_{t+3}\dots e_{r-1}\in B_r^t$ and for
$w\in e_{t+1}e_{t+3}\dots e_{r-1}W_r(\mu_i)$ we have
$$e_{t+1}e_{t+3}\dots e_{r-1}w=\delta^kw \;(r=t+2k),$$
it follows
that $W_r^0(\mu_i)$ generates $W_r(\mu_i)$ as $B_r$-module.
Moreover it is evident that $W_r^0(\mu_i)$ is stable under
$\C\Sym_t\subset B_t\subseteq B_r$, and is isomorphic as
$\C\Sym_t$-module to $S(\mu_i)$.

Since $\phi$ respects the $B_r$-action, we have
$\phi(W_r^0(\mu_1))\subseteq e_{t+1}e_{t+3}\dots e_{r-1}(W_r(\mu_2)/M)$.
It follows that the restriction to $W_r^0(\mu_1)$ of $\phi$ defines a
$\C\Sym_t$-homomorphism $\phi^0:S(\mu_1)\to S(\mu_2)$. But since $\CK$ has
characteristic zero, the cell modules $S(\nu)$ of $\C\Sym_t$ are simple,
whence if $\mu_1\neq \mu_2$, $\phi^0=0$. It follows that $\phi=0$ on
a generating subspace of $W_r(\mu_1)$, whence $\phi=0$.
\end{proof}

Note that Lemma \ref{strictineq} is false if the restriction on
the characteristic of $\CK$ is lifted, since then there may be
non-trivial homomorphisms between the cell (Specht) modules $S(\lambda)$
for $\Sym_t$.

This leads to the following alternative characterisation of the
radical $\Rad_r(\lambda)$ of a cell module of $B_r$.

\begin{cor}\label{char-rad}
Let $\lambda\in\Lambda_r$, with $|\lambda|=t$. For any integer
$m$, let $B_r^m$ be the ideal of $B_r$ spanned by diagrams with
at most $m$ through strings. Then
$$
\Rad_r(\lambda)=\{w\in W_r(\lambda)\mid bw=0\text{ for all }b\in B_r^t\}
=\Ann_{W_r(\lambda)}(B_r^t).
$$
\end{cor}
\begin{proof}
By the previous Lemma, any composition factor of $W_r(\lambda)$ other than $L_r(\lambda)$
is of the form $L_r(\mu)$ with $|\mu|>t$. In particular, any composition factor of
$\Rad_r(\lambda)$ is of this form. But if $|\mu|>t$, then $B_r^t$ acts as zero
on $W_r(\mu)$, since the number of through strings of any diagram $bD$ ($b\in B^t_r$)
has at most $t(>r-|\mu|)$ through strings. It follows that $B_r^tL_r(\mu)=0$.
Hence $B_r^t$ acts as zero on each composition factor of $\Rad_r(\lambda)$, and
hence on $\Rad_r(\lambda)$. So $\Rad_r(\lambda)\subseteq\Ann_{W_r(\lambda)}B^t_r$.

But $\Rad_r(\lambda)$ the unique maximal submodule of $W_r(\lambda)$; in fact
any element of $W_r(\lambda)$ which is not in $\Rad_r(\lambda)$ generates $W_r(\lambda)$
(\cite[Prop. 2.5]{GL96}). Hence we have equality.
\end{proof}

Next, we require the following strengthening of \cite[Thm. 5.4]{DHW}. In the
statement, we use the fact that for any positive integer $r$, we have
$\Lambda_r\subset\Lambda_{r+2}$.

\begin{lem}\label{composn-trans}
Let $\mu,\lambda\in\Lambda_r$. Then $L_r(\mu)$ is a composition factor of
$W_r(\lambda)$ if and only if $L_{r+2}(\mu)$ is a composition factor
of $W_{r+2}(\lambda)$
\end{lem}
\begin{proof} The proof will use the functors $F,G$ introduced in \cite{G}
in the context of Schur algebras, and used in the current context in
\cite{DHW,M}. They are defined as follows. Let $B_r-mod$ denote the category
of left $B_r$-modules. Recalling that $B_{r-2}\subset B_r\subset B_{r+2}$
as described above, define $F:B_r-mod\lr B_{r-2}-mod$ and
$G:B_r-mod\lr B_{r+2}-mod$ by
\be\label{defn:F,G}
\begin{aligned}
&F(M):=e_{r-1}M\\
&G(M):=B_{r+2}e_{r+1}\otimes_{B_r}M,\\
\end{aligned}
\ee
for any $B_r$-module $M$.
The main properties of $F,G$ are that $FG=\id$, $F$ is exact, and $G$
is right exact (cf. \cite[Cor 4.4, p. 143]{HWu}).

To prove the lemma, first note that $L_r(\mu)$ is a composition factor of
$W_r(\lambda)$ if and only if there is a non-trivial homomorphism
$\xi:W_r(\mu)\to W_r(\lambda)/N$ for some submodule $N$ of $W_r(\lambda)$.
Applying the functor $G$, we obtain $G(\xi):G(W_r(\mu))\to G(W_r(\lambda)/N)$,
and by \cite[Lemma 5.1]{DHW}, $G(\xi)\neq 0$. Moreover by [{\it loc.~cit.},
Prop. 5.2], $G(W_r(\mu))=W_{r+2}(\mu)$. To prove that $L_{r+2}(\mu)$
is a composition factor of $W_{r+2}(\lambda)$, it will therefore suffice to
show that $G(W_r(\lambda)/N)\simeq W_{r+2}(\lambda)/M$, for some submodule
$M$ of $W_{r+2}(\lambda)$. But the exact sequence
$W_r(\lambda)\to W_r(\lambda)/N\to 0$ is taken by the
(right exact) functor $G$
to an exact sequence of $B_{r+2}$-modules, whence $G(W_r(\lambda)/N)$
is isomorphic to a quotient of $G(W_r(\lambda))=W_{r+2}(\lambda)$. Hence
if $L_r(\mu)$ is a composition factor of
$W_r(\lambda)$, then $L_{r+2}(\mu)$ is a composition factor of
$W_{r+2}(\lambda)$.

Conversely, if $L_{r+2}(\mu)$ is a composition factor of
$W_{r+2}(\lambda)$, there is a non-trivial homomorphism
$\psi:W_{r+2}(\mu)\to W_{r+2}(\lambda)/M$, for some submodule
$M$ of $W_{r+2}(\lambda)$. Applying the functor $F$, we obtain
a homomorphism
$$F(\psi):F(W_{r+2}(\mu))\to F(W_{r+2}(\lambda)/M).$$
But by the exact nature of $F$, we have $F(W_{r+2}(\lambda)/M)\cong
F(W_{r+2}(\lambda))/F(M)$, and since $F(W_{r+2}(\nu))=W_r(\nu)$
for $\nu\in\Lambda_r$ \cite[Prop. 5.3]{DHW}, we obtain
$F(\psi):W_r(\mu)\to W_r(\lambda)/F(M)$. It remains only to prove that
$F(\psi)\neq 0$. For this, consider the exact sequence
$$
0\lr \Ker \psi\lr W_{r+2}(\mu)\overset{\psi}{\lr}\im \psi\lr 0.
$$
Applying $F$, we obtain an exact sequence
$$
0\lr F(\Ker \psi)\lr W_{r}(\mu)\overset{F(\psi)}{\lr} F(\im \psi)\lr 0,
$$
and hence it suffices to show that $F(\Ker \psi)\neq W_r(\mu)$.
But $\Ker \psi\subseteq\Rad_{r+2}(\mu)=\Ann_{W_{r+2}(\mu)}(B_{r+2}^s)$,
where $|\mu|=s$ (by Cor \ref{char-rad}). It follows that
$e_{s+1}e_{s+3}\dots e_{r+1}\Ker \psi=0=e_{s+1}e_{s+3}\dots e_{r-1}F(\Ker \psi)$.
Moreover since $e_{s+1}e_{s+3}\dots e_{r-1}$ generates the two-sided ideal
$B_r^s$ of $B_r$, it follows that $F(\Ker\Psi)\subseteq\Ann_{W_r(\mu)}(B_r^s)
=\Rad_r(\mu)$.
\end{proof}

\begin{cor}\label{cor:shift}(cf. \cite[Thm. 5.4]{DHW})
Let $\lambda,\mu\in\Lambda_r$, with $|\mu|=s$. Then $L_r(\mu)$ is a
composition factor of $W_r(\lambda)$ if and only if $L_s(\mu)$ is a
composition factor of $W_s(\lambda)$.
\end{cor}
This is immediate from repeated application of Lemma \ref{composn-trans}.

We shall require a combinatorial consequence of Corollary
\ref{cor:shift} which is proved in \cite{DHW}. The statement
involves the following details concerning partitions. Associated to
a partition $\lambda:=(\lambda_1\geq\lambda_2\geq \dots\lambda_p>0)$
we have a ``Young diagram'' which is the set of plane lattice points
$Y(\lambda):=\{(i,j)\in\Z_{\geq 0}\times\Z_{\geq 0}\mid 1\leq j\leq
\lambda_i\}$. For any point $P=(i,j)\in Y(\lambda)$, define the {\it
content} $c(P)$ of $P$ by $c(P)=j-i$. Write $c(\lambda)=\sum_{P\in
Y(\lambda)}c(P)$.

\begin{prop}\label{dhwcrit}(\cite[Thms. 3.1 and 3.3]{DHW})
Let $\mu,\lambda\in\Lambda_r$ with $|\mu|=s$. If $L_s(\mu)$
is a composition factor of $W_s(\lambda)$, then
\begin{enumerate}
\item $Y(\lambda)\subseteq Y(\mu)$ (as sets), and
\item $|\mu|-|\lambda|+\sum_{P\in Y(\mu)\setminus Y(\lambda)}c(P)=0$.
\end{enumerate}
\end{prop}

\begin{cor}\label{composcrit}
Let $\mu,\lambda\in\Lambda_r$ with $|\mu|=s$. If $L_r(\mu)$
is a composition factor of $W_r(\lambda)$, then the
conditions (1) and (2) of Proposition \ref{dhwcrit}
hold.
\end{cor}

This is immediate from Corollary \ref{cor:shift}

\section{Proof of the main theorems.}

Recall that we have the surjective homomorphism (\ref{classeta})
$$
\eta:B_r(3)\lr E(r),
$$
where $E(r)$ is the semisimple $\C$-algebra $\End_{U(\fsl_2)}(V_1^{\otimes r})$.
Write $N=\Ker \eta$; then $\Phi\in N$, so that if $J$ is the two sided ideal
of $B_r(3)$ generated by $\Phi$, we have $J\subseteq N$.
Define the subsets $\Lambda_r^0,\Lambda_r^1$ of $\Lambda_r$ by
$\Lambda_r^0=\{(t), (t-1,1),1^3\mid 0\leq t\leq r;\;\;
t\equiv r(\text{mod $2$})\}$, and $\Lambda_r^1:=\Lambda\setminus\Lambda^0$.

\begin{prop}\label{ann-lambdas}
The following are equivalent conditions on an element $\lambda\in\Lambda_r$.
\begin{enumerate}
\item $\lambda\in\Lambda_r^0$.
\item $NL_r(\lambda)=0$.
\item $JL_r(\lambda)=0$.
\end{enumerate}
\end{prop}
\begin{proof}[Sketch of Proof.]
This statement is contained in \cite[Thm. 6.8]{LZ2}; we give a brief sketch
of the argument for the reader's convenience. One first shows that if $\lambda
\notin\Lambda_r^0$ (i.e. $\lambda\in\Lambda_r^1$) then $\Phi L_r(\lambda)\neq 0$;
this involves explicit computation, and a knowledge of the cases $r=4,5$.
It follows that $N$ acts non-trivially on $L_r(\lambda)$ for $\lambda\in\Lambda_r^1$,
since $\Phi\in N$.

But the semisimple algebra $E(r)$ has $r+1$ simple components, whence $N$ acts as
zero on precisely $r+1$ of the  modules $L_r(\lambda)$. Since $|\Lambda_r^0|=r+1$,
it follows that $N$ acts trivially on $L_r(\lambda)$ for $\lambda\in \Lambda_r^0$.
The statement is now clear.
\end{proof}

\begin{cor}\label{redn:rad}
Let $\CR$ be the radical of $B_r$. Then $N=J$ if and only if $J$ contains $\CR$.
\end{cor}
\begin{proof} We have $E(r)\cong B_r/N$, and
by Prop. \ref{ann-lambdas}, $B_r/N$ and $B_r/J$ have the same maximal semisimple
quotient (viz. $E(r)$). Hence $N=J+\CR$.
\end{proof}

\begin{proof}[Proof of Theorem \ref{mainq}.]

It remains only to show that $J\supseteq \CR$. Since the element $\Phi$ satisfies
$\Phi^*=\Phi$, we may apply Theorem \ref{thm:rad} above, which implies that it
suffices to show that for $\lambda\in\Lambda_r^0$, $JW_r(\lambda)=\Rad_r(\lambda)$.
Note that by Prop. \ref{ann-lambdas}, we have for $\lambda\in\Lambda_r^0$,
$JL_r(\lambda)=0$, whence $JW_r(\lambda)\subseteq\Rad_r(\lambda)$. We shall
prove that we have equality by showing, for $\lambda\in\Lambda_r^0$,
\be\label{crux}
{\text{If $L_r(\mu)$ is a composition factor of $W_r(\lambda)$ and
$\mu\neq\lambda$, then }}\mu\in\Lambda_r^1.
\ee

Given (\ref{crux}), it follows that for $\lambda\in\Lambda_r^0$, all composition
factors $L_r(\mu)$ of $\Rad_r(\lambda)$ satisfy $\mu\in\Lambda_r^1$, whence
$JL_r(\mu)=L_r(\mu)$. It follows that $J\Rad_r(\lambda)=\Rad_r(\lambda)$, whence by
Theorem \ref{thm:rad}, $J\supseteq \CR$.

It therefore remains only to prove (\ref{crux}). We do this by invoking the
criteria in Prop. \ref{dhwcrit} for $L_r(\mu)$ to be a composition factor
of $W_r(\lambda)$. There are three cases.

First, if $\lambda=(t)$ and $\mu\in\Lambda_r^0$ with $Y(\mu)\supset
Y(\lambda)$, then there are three possibilities for $\mu$: (a)
$\mu=(s)$, $s>t$; (b) $\mu=(s-1,1)$, $s>t$; (c) $t=1$ and
$\mu=(1^3)$. In cases (a) and (b), $|\mu|-|\lambda|+\sum_{P\in
Y(\mu)\setminus Y(\lambda)}c(P)>0$, while in case (c),
$|\mu|-|\lambda|+\sum_{P\in Y(\mu)\setminus Y(\lambda)}c(P)=-1\neq
0$.

Secondly, if $\lambda=(t-1,1)$, the only possibility for
$\mu\in\Lambda_r^1$ with $Y(\mu)\supset Y(\lambda)$ is $\mu=(s-1,1)$
with $s>t$. In this case, again $|\mu|-|\lambda|+\sum_{P\in
Y(\mu)\setminus Y(\lambda)}c(P)>0$.

Finally, if $\lambda=(1^3)$, there is no $\mu\in\Lambda_r^0$ with $Y(\mu)\supset
Y(\lambda)$.

This proves Theorem \ref{mainclass}, and hence by Prop. \ref{redn:class}
completes the proof of Theorem \ref{mainq}.
\end{proof}

\section{An integral form for the endomorphism algebra.}
Let $\CA$ be the ring $\C[q^{\pm 1}]$,
%and let $\CA_q$ be the
%localisation of $\CA$ at the multiplicative set generated by
%the set $\{[3]_q=q^2+1+q^{-2}, q+q\inv,q^4-q^{-4}\}$.
Lusztig \cite{L1,L2}
has defined an $\CA$-form $\cU_\CA$ of the quantised
universal enveloping algebra $\cU_q(\fg)$ of a Lie algebra $\fg$, which involves the
divided powers of the generators $e_i,f_i$ of $\cU_q$. He also showed
how to construct $\CA$-forms of higest weight modules for $\cU_\CA$
by applying these divided powers to highest weight vectors of the
corresponding $\cU_q$-modules. Define $\phi_1:\CA\to\C$
by $\phi_1(q)=1$ and let $\phi_q:\CA\to\CK$ be the inclusion map.
%By taking $-\otimes_\CA \CA_q$ of the above spaces, we obtain
%the corresponding statements with $\CA$ replaced by $\CA_q$

We shall define an $\CA$-form of the
$\CK$-algebra $E_r(2,q)\cong BMW_r(\CK)/\langle \Phi_q\rangle$
which specialises to the endomorphism algebras $E_r(2)$ and $E_r(2,q)$
respectively when we map $\CA$ to $\C$ and to $\CK$ via $\phi_1$ and $\phi_q$
respectively.

The definition is as follows. Let $BMW_r(\CA)$ be the $\CA$-algebra generated
by the set $\{g_1^{\pm 1},\dots g_{r-1}^{\pm 1};e_1,\dots,e_{r-1}\}$
subject to the relations
(\ref{braidgiq}). It is explained in \cite[p.285]{X} that $BMW_r(\CA)$ is
free as an $\CA$-module, with basis a set of ``tangle diagrams''.

Let $a,b,c$ and $d$ be the elements of $\CK$ defined in (\ref{abcd}),
and write $\tilde a=(q^2+q^{-2})a$, $\tilde b=(q^2+q^{-2})b$,
$\tilde c=(q^2+q^{-2})c$ and $\tilde d=(q^2+q^{-2})d$.
Then $\tilde a,\tilde b,\tilde c,\tilde d\in\CA$,
and we define
$\tilde \Phi_q$ as
$\tilde \Phi_q=(q^2+q^{-2})\Phi_q$, i.e. (cf. (\ref{defPhiq})
$$
\tilde\Phi_q =\tilde aF_q e_2 F_q - \tilde b F_q -
\tilde c F_qe_2e_{14}F_q +\tilde d F_q e_{1 2 3 4} F_q.
$$

Then $\tilde\Phi_q\in BMW_r(\CA)$.

\begin{definition}\label{def:er} Define the $\CA$-algebra $\CE_r(\CA)$ by
$\CE_r(\CA):=BMW_r(\CA)/\langle \tilde\Phi_q \rangle$.
\end{definition}

\begin{prop}\label{specend}
Let $\CE_r(\CA)$ be the $\CA$-algebra of (\ref{def:er}), and let $\phi_1:\CA\to\C$
and $\phi_q:\CA\to\CK$ be the homomorphisms defined above, which make $\C$ and $\CK$
into $\CA$-modules. Then we have isomorphisms
$$
\CE_r(\CA)\otimes_{\CA}\C\overset{\simeq}{\lr}\End_{\cU(\fsl_2)}V_1^{\otimes r},
$$
and
$$
\CE_r(\CA)\otimes_{\CA}\CK\overset{\simeq}{\lr}\End_{\cU_q(\fsl_2)}V_q^{\otimes r},
$$
where $V_1$ is the irreducible $\fsl_2(\C)$ module of highest weight $2$ (dimension $3$),
and $V_q$ is its $\cU_q$-analogue.
\end{prop}
\begin{proof}
We prove the second statement. the proof of the first is similar.
We have an exact sequence of $\CA$-algebras
\be
0\lr \langle\tilde\Phi_q\rangle\lr BMW_r(\CA)\lr\CE_r(\CA)\lr 0.
\ee
Moreover we know that $BMW_r(\CA)$ is free as $\CA$-module, whence
so is the kernel $\langle\tilde\Phi_q\rangle$.
Applying the functor $-\otimes_{\CA}\CK$, which by (cf. \cite[Cor 4.4, p. 143]{HWu})
is right exact, we obtain an exact sequence
\be
BMW_r(\CK)\lr\CE_r(\CA)\otimes_{\CA}\CK\lr 0,
\ee
and the kernel of the first map contains $\Phi_q$.

%Now Theorem \ref{mainq} implies that
%$\CE_r(\CA)$ is the $\CA$-algebra of endomorphisms of Lusztig's $\CA$-form
%of $V^{\otimes r}$ which are induced by the $R$-matrices.

Now since $\tilde\Phi_q$ acts trivially on Lusztig's $\CA$-form
of $V^{\otimes r}$, it follows that there is a homomorphism
$\CE_r(\CA)\lr\End_{\cU_\CA(\fsl_2)}(V_\CA^{\otimes r})$, whose image
is the algebra of endomorphisms which are generated by the $R$-matrices,
and whose kernel is, by Theorem \ref{mainq}, a torsion submodule.
The statement follows.
\end{proof}

\section{Interpretation in terms of orthogonal Lie algebras.}

It is well known that the Lie algebras $\fsl_2(\C)$ and
$\fso_3(\C)$ are abstractly isomorphic. In this section we shall
make explicit the interpretation of our main theorem in terms of $\fso_3$.
We shall discuss the classical ($q=1$) situation; the quantum
case may be discussed entirely similarly.

The three-dimensional representation of $\fsl_2(\C)$ may be realised
as follows. Let $V=\C^3$ be the subspace of the polynomial ring
$\C[x,y]$ with basis $x^2,xy,y^2$. Then $\fsl_2(\C)$ acts via
$e=x\frac{\partial}{\partial y}$, $f=y\frac{\partial}{\partial x}$
and $h=[e,f]=x\frac{\partial}{\partial x}-y\frac{\partial}{\partial y}$.
With respect to the given basis, $e,f$ and $h$ are represented by the
matrices
$\left[\begin{smallmatrix}0&1&0\\0&0&2\\0 &0 & 0\\ \end{smallmatrix}
\right]$,
$\left[\begin{smallmatrix}0 & 0 & 0\\2 & 0 & 0\\0 & 1 & 0\\ \end{smallmatrix}
\right]$ and
$\left[\begin{smallmatrix}2 &0 & 0\\0&0 &0\\0&0&-2\\ \end{smallmatrix}
\right]$.

Now $\fso_3(\C)$ may be identified with the space of skew symmetric
matrices. If we write $(a,b,c)$ for the matrix
$\left[\begin{smallmatrix}0&-a &-b\\a& 0 &-c\\b&c&0\\ \end{smallmatrix}
\right]$, the Lie product $[(a_1,a_2,a_3), (b_1,b_2,b_3)]=(c_1,c_2,c_3)$,
where $c_i=(-1)^{i-1}\det M_i$, where $M_i$ is the $2\times 2$ matrix
obtained from
$\left[\begin{smallmatrix}a_1&a_2&a_3\\b_1&b_2&b_3\\ \end{smallmatrix}
\right]$ by deleting the $i^{\text{th}}$ column.
Writing $E=(i,0,1),F=(i,0,-1)$ and $H=(0,2i,0)$, where $i=\sqrt -1$,
it is easily verified that $E,F$ and $H$ satisfy the $\fsl_2$ relations,
and that
$T=\left[\begin{smallmatrix}1&0&1\\0&i&0\\i&0&-i\\
\end{smallmatrix}\right]$
conjugates $e,f,h$ into $E,F,H$ respectively, i.e. that $TeT\inv=E$ etc.

Inverting $T$, it follows that with respect to the basis
$w_1=x^2+y^2$, $w_2=-2ixy$, $w_3=-ix^2+iy^2$, $e,f$ and $h$ have
skew symmetric matrices. This identifies the 3-dimensional
representation of $\fsl_2(\C)$ explicitly with the natural
representation of $\fso_3(\C)$. The quantum case is similar.

It follows that our main theorems (\ref{mainclass} and \ref{mainq})
may be stated in terms of the natural representation of
$\fso_3$. Now the natural representation $V_n$ of
quantum $\fso_{2n+1}$ shares
with the case $n=1$ the properties
that it is strongly multiplicity free,
and that the surjection $\C(q)\cB_r
\to\End_{\cU_q(\fso_{2n+1})}(V_n^{\otimes r})$, where $\cB_r$ is the $r$-string
braid group, factors through a specialisation of the
algebra $BMW_r(y,z)$. It may be reasonable to speculate that
a result similar to ours holds in this generality.

\end{document}